\documentclass[11pt]{article}

\pdfoutput=1  

\usepackage{geometry}
\usepackage{cite,url}

\usepackage[usenames,dvipsnames]{color}
\usepackage{graphicx,wrapfig}

\usepackage[skip=3pt,font=scriptsize]{caption}

\usepackage[font={scriptsize}]{subcaption}

\usepackage[cmex10]{amsmath}
\usepackage{amsfonts,amssymb,amsthm,mathrsfs}
\usepackage[colorlinks,bookmarksopen,bookmarksnumbered,citecolor=red,urlcolor=red]{hyperref}
\usepackage{soul}

\newlength{\noteWidth}
\long\def\notes#1{\ifinner
           {\footnotesize #1}
           \else
           \marginpar{\parbox[t]{\noteWidth}{\raggedright\footnotesize #1}}
       \fi\typeout{#1}}

\def\notes#1{\typeout{read notes: #1}}  

\def\spm#1{\notes{SPM:  #1}}

\def\rd#1{{\color{red}#1}}

\def\ToDo#1{\medbreak \rd{\noindent\framebox{\parbox{.95\hsize}{\raggedright {\bf To do:} #1 }}}\medbreak}

\def\ToDo#1{\typeout{read to do entries}}

\newcommand{\bsq}{\rule{1.25ex}{1.25ex}}

\def\bqed{\ifmmode\bsq\else{\unskip\nobreak\hfil
\penalty50\hskip1em\null\nobreak\hfil\bsq
\parfillskip=0pt\finalhyphendemerits=0\endgraf}\fi\medskip}

\newenvironment{DesignCDC}[1]{\null\smallbreak \noindent \textbf{#1} \nobreak \ }{\null\par}

\def\wham#1{\smallbreak\noindent\textbf{#1}}

\def\rDD{DD-r}
\def\xz{x^\circ}

\def\tilutilB{\tilutil_{B_\zeta}} 

\def\He{H_{\text{\rm\scriptsize e}}}
\def\preHe{H_{\text{\rm\scriptsize e}}^\circ}
\def\TF{G}

\def\EF{h}
\def\EFn#1{\Lambda_{#1}}

\def\preclH{{\cal H}^\circ}
\def\preH{H^\circ}  

 \def\adj{\Lsh}   
 \def\adjnat{\triangledown} 

\def\urls#1{{\small \url{#1}}}

\def\sq{\hbox{\rlap{$\sqcap$}$\sqcup$}}
\def\qed{\ifmmode\sq\else{\unskip\nobreak\hfil
\penalty50\hskip1em\null\nobreak\hfil\sq
\parfillskip=0pt\finalhyphendemerits=0\endgraf}\fi\medskip}

\long\def\defbox#1{\framebox[.9\hsize][c]{\parbox{.85\hsize}{%
\parindent=0pt
\baselineskip=12pt plus .1pt      
\parskip=6pt plus 1.5pt minus 1pt 
 #1}}}

\long\def\beginbox#1\endbox{\subsection*{}%
\hbox{\hspace{.05\hsize}\defbox{\medskip#1\bigskip}}%
\subsection*{}}

\def\endbox{}

\def\transpose{{\hbox{\it\tiny T}}}

\newsavebox{\junk}
\savebox{\junk}[1.6mm]{\hbox{$|\!|\!|$}}

\def\argmax{\mathop{\rm arg\, max}}

\def\state{{\sf X}}
\def\stateu{{\sf X}_{\sf u}}
\def\staten{{\sf X}_{\sf n}}

\newcommand{\field}[1]{\mathbb{#1}}

\def\Re{\field{R}}

\def\ind{\field{I}}

\def\cP{{\check{P}}}

\def\bfmath#1{{\mathchoice{\mbox{\boldmath$#1$}}%
{\mbox{\boldmath$#1$}}%
{\mbox{\boldmath$\scriptstyle#1$}}%
{\mbox{\boldmath$\scriptscriptstyle#1$}}}}

\def\bfmm{\bfmath{m}}

\def\bfmX{\bfmath{X}}

\def\bfmY{\bfmath{Y}}

\def\bfmhhaY{\bfmath{\hhaY}} 
\def\bfmhhaY{\hbox to 0pt{$\widehat{\bfmY}$\hss}\widehat{\phantom{\raise 1.25pt\hbox{$\bfmY$}}}}

\def\bfzeta{\bfmath{\zeta}}

\def\til={{\widetilde =}}

\def\clE{{\cal E}}

\def\clH{{\cal H}}

\def\clI{{\cal I}}

\def\clW{{\cal W}}

 \def\FRAC#1#2#3{\genfrac{}{}{}{#1}{#2}{#3}}

\def\ddtp{{\mathchoice{\FRAC{1}{d^{\hbox to 2pt{\rm\tiny +\hss}}}{dt}}%
{\FRAC{1}{d^{\hbox to 2pt{\rm\tiny +\hss}}}{dt}}%
{\FRAC{3}{d^{\hbox to 2pt{\rm\tiny +\hss}}}{dt}}%
{\FRAC{3}{d^{\hbox to 2pt{\rm\tiny +\hss}}}{dt}}}}

\def\ddzeta{{\mathchoice{\FRAC{1}{d}{d\zeta}}%
{\FRAC{1}{d}{d\zeta}}%
{\FRAC{3}{d}{d\zeta}}%
{\FRAC{3}{d}{d\zeta}}}}

\def\eqdef{\mathbin{:=}}

\def\Prob{{\sf P}}

\def\Expect{{\sf E}}

\def\average#1,#2,{{1\over #2} \sum_{#1}^{#2}}

\def\eye(#1){{\bf(#1)}\quad}

\newtheorem{theorem}{Theorem}[section]

\newtheorem{proposition}[theorem]{Proposition}
\newtheorem{lemma}[theorem]{Lemma}

\def\Lemma#1{Lemma~\ref{#1}}
\def\Prop#1{Prop.~\ref{#1}}
\def\Theorem#1{Theorem~\ref{#1}}

\def\Table#1{Table~\ref{#1}}
\def\Section#1{Section~\ref{#1}}

\def\eq#1/{(\ref{e:#1})}

\newcommand{\beqn}[1]{\notes{#1}%
\begin{eqnarray} \elabel{#1}}

\newcommand{\eeqn}{\end{eqnarray} }

\newcommand{\beq}[1]{\notes{#1}%
\begin{equation}\elabel{#1}}

\newcommand{\eeq}{\end{equation}}

\def\bdes{\begin{description}}
\def\edes{\end{description}}

\newcounter{rmnum}
\newenvironment{romannum}{\begin{list}{{\upshape (\roman{rmnum})}}{\usecounter{rmnum}
\setlength{\leftmargin}{8pt}
\setlength{\rightmargin}{6pt}
\setlength{\itemsep}{2pt}
\setlength{\itemindent}{-1pt}
}}{\end{list}}

\newcounter{anum}

\def\bfDelta{\bfmath{\Delta}}

\def\util{\mathchoice{\mbox{\small$\cal U$}}%
{\mbox{\small$\cal U$}}%
{\mbox{$\scriptstyle\cal U$}}%
{\mbox{$\scriptscriptstyle\cal U$}}}

\def\tilutil{\mathchoice{\mbox{\small$\cal \widetilde U$}}%
{\mbox{\small$\cal\widetilde U$}}%
{\mbox{$\scriptstyle\cal \widetilde U$}}%
{\mbox{$\scriptscriptstyle\cal \tilde U$}}}

\def\meanutil{\mbox{\scriptsize$\bar{\cal U}$}}

\def\Ebox#1#2{%
\begin{center}
\includegraphics[width= #1\hsize]{#2} \end{center}}

\graphicspath{{figures/}}

\def\Fig#1{Fig.~\ref{#1}}

\def\ind{\field{I}}

\def\Re{\field{R}}

\title{Distributed Randomized Control \\
 for Demand Dispatch
}

\author{Ana Bu\v{s}i\'{c} and Sean Meyn
\thanks{Research   supported by
French National Research Agency grant ANR-12-MONU-0019,  and  NSF grants CPS-0931416 and CPS-1259040.}
\thanks{A.B.\ is with Inria and the Computer Science Dept. of \'Ecole Normale Sup\'erieure, Paris, France;
S.M. is with the Department of Electrical and Computer
Engg.\ at the University of Florida, Gainesville. }%
}

\begin{document}

\maketitle
\thispagestyle{empty}

\begin{abstract} 

The paper concerns design of control systems for \textit{Demand Dispatch}  to obtain ancillary services to the power grid by harnessing inherent flexibility in many loads. The role of ``local intelligence'' at the load has been advocated in prior work; randomized local controllers that manifest this intelligence are convenient for loads with a finite number of states.   
\notes{ With careful design,  the grid operator can harness this flexibility to regulate supply-demand balance.   The deviation in aggregate power consumption can be controlled just as generators provide ancillary service today.  This is achieved without loss of quality of service to the end consumers. } 
The present work introduces two new design techniques for these randomized controllers:
\begin{romannum}
\item 
The \textit{Individual Perspective Design} (IPD)  is based on the solution to a one-dimensional family of Markov Decision Processes, whose objective function is formulated from the point of view of a single load.   The  family of dynamic programming equation appears complex, but it is shown that it is obtained through the solution of a single ordinary differential equation.

\item 
The \textit{System Perspective Design} (SPD)  is motivated by a single objective of the grid operator:  Passivity of any linearization of the aggregate input-output model.  A solution is obtained that can again be computed through the solution of a single ordinary differential equation.  
\end{romannum}
Numerical results complement these theoretical results.

\end{abstract}



\section{Introduction} 
\label{s:intro}

\spm{
Work on intro, especially prior work, and Kalman stuff}

Renewable energy sources such as wind and solar power have a high degree of unpredictability and time variation, which makes balancing demand and supply increasingly challenging. One  way to address this challenge is to harness the inherent flexibility in demand of many types of loads.   Loads can supply a range of ancillary  services to the grid, such as the balancing reserves required at Bonneville Power Authority (BPA),  or the Reg-D/A regulation reserves used at PJM \cite{barbusmey14}.  Today these services are secured by a balancing authority (BA) in each region.

\subsection{Demand dispatch}

These grid services can be obtained without impacting quality of service (QoS)  for consumers \cite{chebusmey14,barbusmey14}, but this is only possible through design.
The term \textit{Demand Dispatch} is used in this paper 
to emphasize the difference between the goals of our own work and traditional demand response.  
\spm{Please see cite{brolureispiwei10} and summarize what they have actually done}

\begin{figure}[h]
\Ebox{.5}{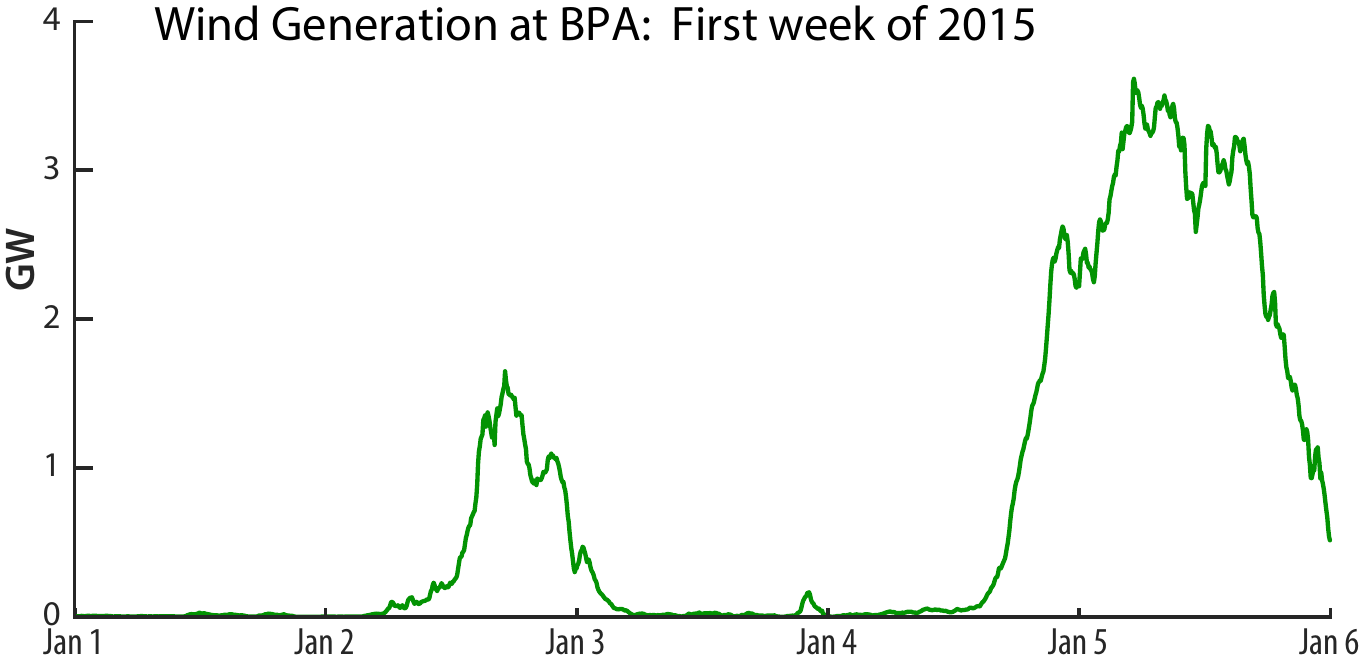}  
\caption{Wind generation at BPA.}
\label{f:windHICSS}
\vspace{-1.25ex}
\end{figure} 

\spm{stole text from our proposal and my SIAM article to define VES:}
Consumers use power for a reason, and expect some guarantees on the QoS they receive. 
\spm{axe for CDC: For example, the temperature in a building or refrigerator must lie within strict bounds. } 
The grid operator desires  reliable ancillary service,
obtained from the inherent flexibility in the consumer's power consumption.   
These seemingly conflicting goals can be achieved simultaneously, but the solution requires local control:  an appliance must monitor its QoS and other state variables, it  
must receive grid-level information (e.g., from the BA),  and based on this information it must  adjust its power consumption.  With proper design,  an aggregate of loads can be viewed at the grid-level as \textit{virtual energy storage} (VES).  Just like a battery,  the aggregate  provides ancillary service, even though it cannot produce energy.

\Fig{f:windHICSS} shows the wind generation in the BPA region during the first week of 2015.   There is virtually no power generated on New Year's Day,  and generation ramps up to nearly 4GW on the morning of January 5.  
\notes{... There is also volatility introduced from loads and from neighboring balancing authorities.} 
In this example we show how to supply a demand of exactly 4GW during this time period, using  generation from wind and other resources.

\begin{figure}[h]
\Ebox{.5}{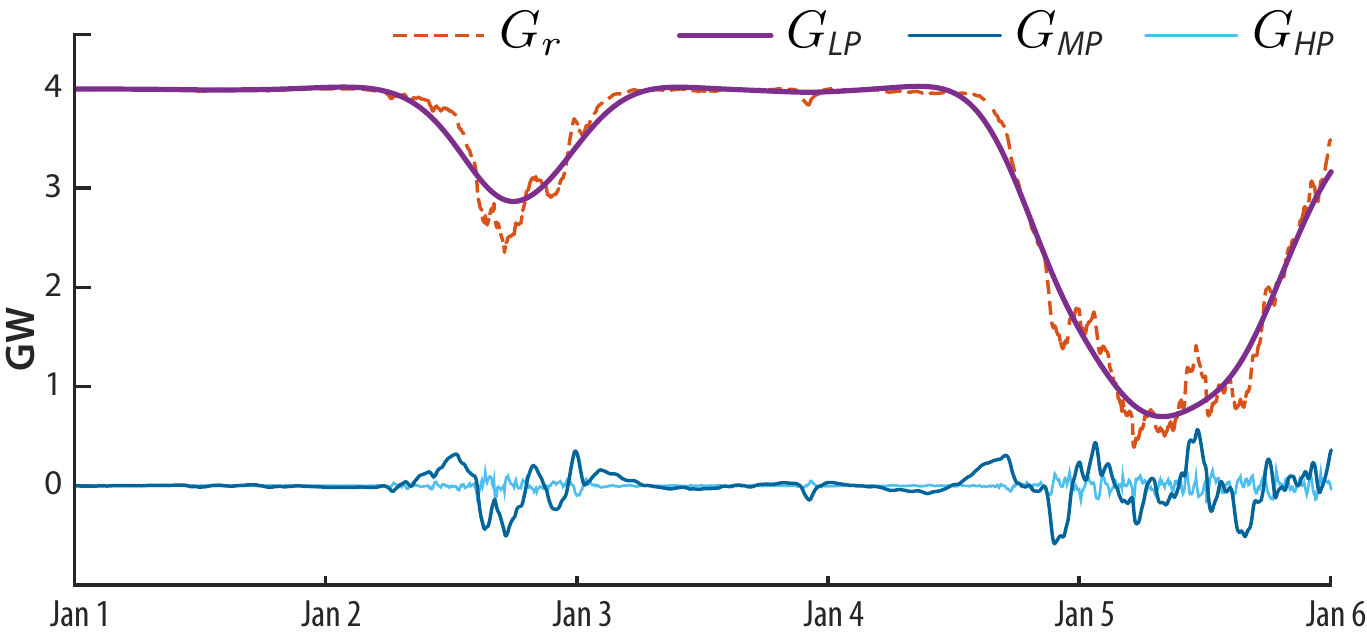}  
\caption{One power signal representing significant energy delivery, and two zero-energy power signals combine with wind generation to match a 4GW load. }
\label{f:3ProductsHICSS}
\vspace{-1.25ex}
\end{figure}

Let $G_r(t)$ denote the additional power required at time $t$, in units of GWs.  For example, on the first day of this week we have $G_r(t)\approx 4$.    Day-ahead forecast of the low frequency component of generation from wind is highly predictable.   
We let $G_{\text{\it LP}}$ denote the signal obtained by passing the forecast of $G_r$ through a low pass filter.   The signal $G_{\text{\it HP}}$ is obtained by filtering $G_r- G_{\text{\it LP}}$ using a high pass filter,  and $G_{\text{\it MP}} = G_r - G_{\text{\it LP}}-G_{\text{\it HP}}$.   Each of these filters is causal, and their parameters  are a design choice.
Examples are shown in \Fig{f:3ProductsHICSS}.

 It is not difficult to ramp hydro-generation up and down to accurately track the  power signal $G_{\text{\it LP}}$.
 This is an energy product that might be secured in today's day-ahead markets. 
 
The other two signals shown in  \Fig{f:3ProductsHICSS} take on positive and negative values.  Each represents a total energy of approximately zero, hence it would be a mistake to attempt to obtain these services in an energy market.   Either could be obtained from a large fleet of batteries or flywheels.  However, it may be much cheaper to employ flexible loads via demand dispatch.   The signal $G_{\text{\it HP}}$ can be obtained by modulating  the fans in commercial buildings (perhaps by less than 10\%)  \cite{haolinkowbarmey14}.    The signal  $G_{\text{\it MP}} $ can be supplied in whole or in part by loads such as water heaters, commercial refrigeration, and water chillers.

Low frequency variability from solar gives rise to the famous ``duck curve'' anticipated at CAISO\footnote{an ISO in California: \urls{www.caiso.com}},  which is represented as the hypothetical ``net-load curve'' in \Fig{f:duck}.     The actual net-load curve is the difference between load and generation from renewables;  the drop from 20 to 10 GW is expected with the introduction of 10 GW of solar in the state of California.  This curve highlights the ramping pressure placed on conventional generation resources.  

As shown in this figure, the volatility and steep ramps associated with California's duck curve can be addressed using a frequency decomposition:  The plot shows how the net-load can be expressed as the sum of four signals distinguished by frequency.    Variability introduced by the low frequency component  can be rejected using traditional resources such as thermal generators,  along with some flexible loads (e.g., from flexible industrial manufacturing).  The mid-pass signal shown in the figure would be a challenge to generators for various reasons, but this zero-mean signal, as well as the higher frequency components, can be tracked using a variety of flexible loads.

The control architecture described in this paper is not limited to handling disturbances from wind and solar energy.  \Fig{f:ContingencyRecharge} illustrates how the same frequency decomposition can be used to allocate resources following an unexpected contingency, such as a generator outage.

\begin{figure*}[h!]
    \centering
    \begin{subfigure}[t]{0.45\textwidth}
        \centering
        \includegraphics[width=\hsize]{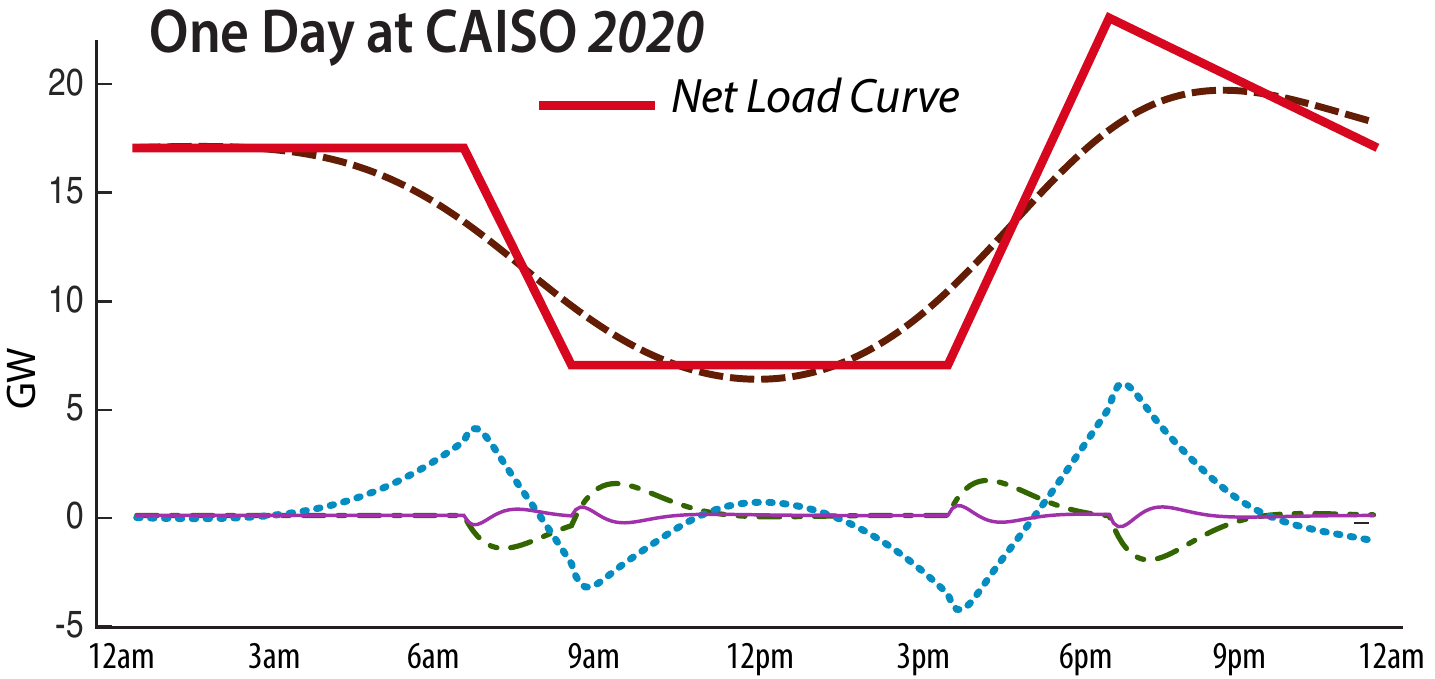}
        \caption{The  ``duck curve''  is decomposed as the sum of four  signals:
 The three high-frequency  components can be obtained using VES.}
         \label{f:duck}
    \end{subfigure}%
\hfil
 \begin{subfigure}[t]{0.45\textwidth}
        \centering
        \includegraphics[width=\hsize]{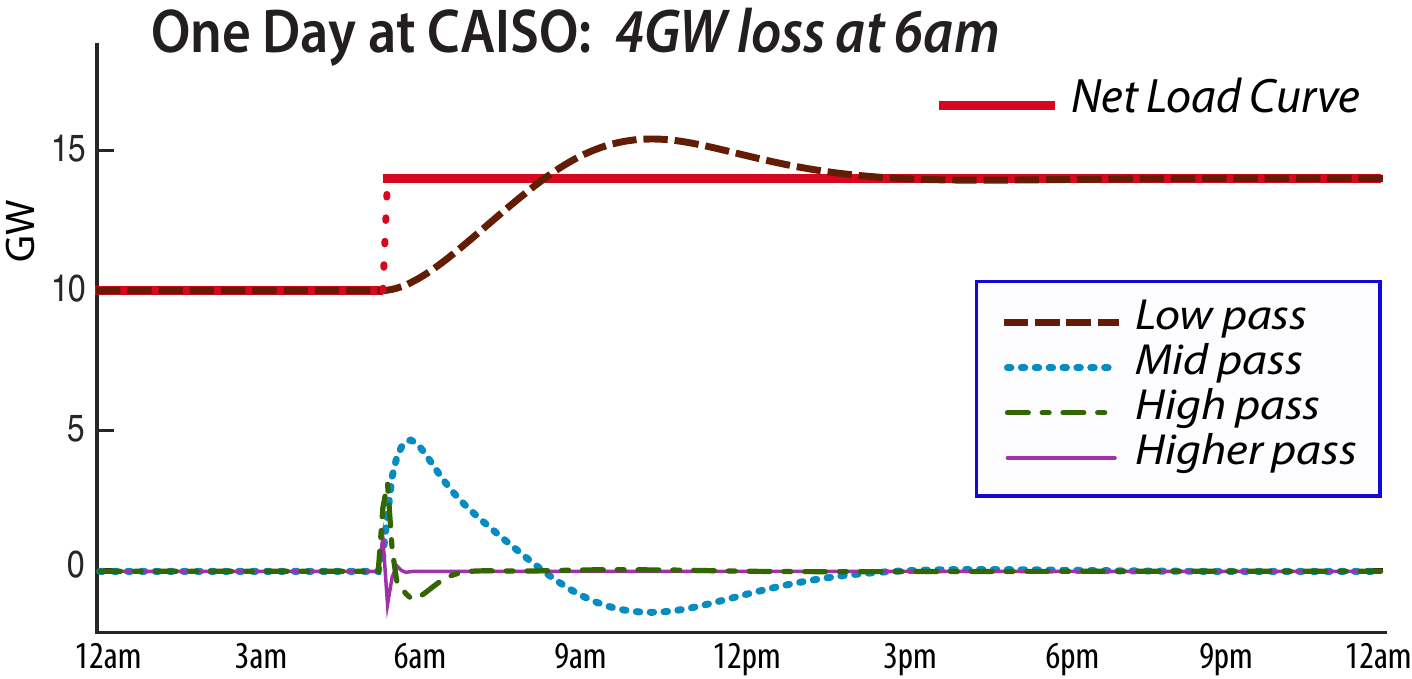}
        \caption{An unexpected power outage requires quick ramping of power generation.     The  zero-energy components are obtained using VES.}
        \label{f:ContingencyRecharge}
    \end{subfigure}
    \smallskip
    
    \caption{Frequency decomposition in two examples in which large ramps in generation are required. In each case, the low-pass energy component can be obtained through conventional generation, or load-shedding.  The remaining components of net-load are obtained using VES.}
     \label{f:freqDecomp}
\vspace{-1.5ex}
\end{figure*}

For loads whose power consumption cannot be varied continuously,  we have argued in prior work that a distributed randomized control architecture is convenient for design  \cite{meybarbusyueehr15,chebusmey14,chebusmey15b}.  
\spm{Ana feels "varied continuously" should be revised.  It is time to submit, and I can't think of a revision}
This architecture includes local control to maintain bounds on the quality of service delivered by the loads,  and also to ensure high quality ancillary service to the grid.  

Analysis of the aggregate is based on  a mean-field model.

\subsection{Mean-field model}

We restrict to the setting of the prior work  \cite{meybarbusyueehr15,chebusmey14}, based on the following  distributed control architecture.  A family of transition matrices $\{P_\zeta :\zeta\in\Re\}$ is constructed to define local decision making.   Each load evolves as a controlled Markov chain on a finite state space, with common input $\bfzeta=(\zeta_0,\zeta_1,\dots)$.  it is assumed that the scalar signal $\bfzeta$ is broadcast to each load.   If a load is in state $x$ at time $t$,  and the value $\zeta_t$ is broadcast,  then the load transitions to the state $x'$ with probability $P_{\zeta_t}(x,x')$.  Letting $X^i_t$ denote the state of the $i$th load at time $t$,  and assuming $N$ loads, the empirical pdf   (probability mass function) is defined as the average,
\[
\mu^N_t(x) = \frac{1}{N} \sum_{i=1}^N \ind\{ X^i_t =x\},\qquad x\in\state.
\]
The mean-field model is the deterministic system defined by the evolution equations,
\begin{equation}
\mu_{t+1} =\mu_t P_{\zeta_t},\quad t\ge 0,
\label{e:MFJ}
\end{equation}
in which $\mu_t$ is a row vector.   Under general conditions on the model and on $\mu_0$ it can be shown that $\mu^N_t$ is approximated by $\mu_t$.

In this prior work it is assumed that  average power consumption is obtained through measurements or state estimation:
Assume that 
$\util(x)$ is the power consumption when the load is in state $x$,
for some function $\util\colon\state\to\Re_+$.  The average power consumption is defined by,
\[
y^N_t =  \frac{1}{N} \sum_{i=1}^N  \util( X^i_t),\qquad x\in\state.
\]
which is approximated using the mean-field model: 
\begin{equation}
y_t = \sum_x \mu_t(x)\util(x),\quad t\ge 0.
\label{e:obs}
\end{equation}
The mean-field model is a  state space model that is linear in the state $\mu_t$, and nonlinear in the input $\zeta_t$.   The observation  process  \eqref{e:obs} is also linear as a function of the state.  
Assumptions imposed in the prior work
 \cite{chebusmey14,chebusmey15b,meybarbusyueehr15} imply that the input is a continuous function of these values.
 

In \cite{meybarbusyueehr15}, the design of the feedback law 
$\zeta_t = \phi_t(y_0,\dots, y_t)$ is based on a linearization of this state space model.   One goal of the present paper is to develop design techniques to ensure that the linearized input-output model has desirable properties for control design at the grid level.

\subsection{Contributions}

Several new design techniques are introduced in this paper, and the applications go far beyond prior work:
\begin{romannum}
\item  \textit{Optimal design}.
In the prior work \cite{meybarbusyueehr15}, the family of transition matrices $\{P_\zeta\}$ was constructed  based on an average-cost optimal control problem (MDP).   The cost function in this MDP was parameterized by the scalar $\zeta$.  
In this prior work, the optimal control problem was completely unconstrained,  in the sense that any choice of $P$ was permissible in the optimal control formulation.
\spm{axe:  If some randomness is exogenous, then this approach leads to an infeasible solution --- an optimal solution may require the load to have the ability to change the weather!  }
The optimal control formulation proposed here is far more general:  We allow some randomness by design, and some exogenous randomness that is beyond our control.  These contributions are summarized in   \Theorem{t:IPD}.

\notes{Not the main result any more:
\\
The dynamic programming equations become significantly  more complex,  but remain tractable.    \Theorem{t:IPD} summarizes the main results:  \textit{It is found that the solution to the entire family of MDPs can be obtained through the solution of a single ordinary differential equation {\rm (ODE)}.}   
}

\notes{I don't think this is necessary here:
It would be desirable to the BA for a collection of similar loads  to behave like an ideal wire -- without resistance or delay.  This is not possible because all physical systems have temporal dynamics. 
The following approximation to this idealization is called   \textit{passivity}.  Since we consider passivity of the linearization, this is equivalent to the following property of its transfer function:   ...}

\item \textit{Passivity by design}.
A discrete-time transfer function $F$ is \textit{positive-real} if it is stable (all poles are strictly within the unit disk), and the following bound holds:
\[
F(e^{j\theta} )+ F(e^{-j\theta} )\ge 0,\quad \theta\in\Re
\]
It is \textit{strictly positive-real} if the inequality is strict for all $\theta$.  A linear system is passive if it is positive real.
\spm{Ana asks if the relevance is clear to these readers.   I will think about how to give
concise motivation}

Let  $(A,B,C)$ describe a state-space representative of the linearization, with transfer function $\TF(z) = C(Iz-A]^{-1}B$.
Consider  the delay-free model with transfer function $\TF^+(z) \eqdef z \TF(z)$.  That is,
\begin{equation}
\TF^+(z)  =zC[Iz-A]^{-1} B = \sum_{k=0}^\infty C A^k B z^{-k}
\label{e:G+}
\end{equation}
A new approach to design of the transition matrices is introduced in this paper to ensure that the linearization is strictly positive real.  The main conclusions are summarized in \Theorem{t:SPD}.   

\item \textit{ODE methods for design}.  A unified  computational framework is introduced.
The construction of the transition matrices $\{P_\zeta :\zeta\in\Re\}$ is  obtained as the solution to a single ODE for each of the  design techniques  (i) and (ii).

\item \textit{Applications.}   Prior work on distributed control for demand dispatch focused on a single collection of loads:  residential pool pumps \cite{meybarbusehr13,meybarbusyueehr15,chebusmey15b}.   The motivation was ease of exposition, and also the fact that the design methodology required special assumptions on nominal behavior.  The new methodology developed in this paper relaxes these assumptions, and allows application to any load with discrete power states, such as a refrigerator, or other thermostatically controlled loads (TCLs).

\end{romannum}

\subsection{Prior work}

\notes{Need lots on prior work. refs to 
  \cite{huacaimal07,gasgauleb12}
  and clear explanation of why we are different from all others.  NOBODY CAN TRACK LIKE WE CAN!}

There are many recent papers with similar goals -- to create a science to support demand dispatch.    In \cite{matkoccal13} and its sequels, all control decisions are at the balancing authority,  and this architecture then requires state estimation to obtain the grid-level control law.  A centralized deterministic approach is developed in \cite{sanhaopoo14,biehanandsto13}.    None of this prior work considers design of local control algorithms, which is the focus of this paper.  
%
%
%
%

Passivity was established in   the prior work  \cite{busmey14}, but only for continuous time models for which the nominal model (with $\zeta=0$) is a \textit{reversible} Markov process.   
\spm{axe: This is a very strong requirement on the model.  In the present work we \textit{do not} assume reversibility.  Rather, we show that it is possible to design $\{P_\zeta :\zeta\in\Re\}$  so that $\TF^+$ is strictly positive-real. } 
It follows that $\TF^+$ 
is minimum phase, and hence the
 original transfer function $\TF(z) =z^{-1}\TF^+(z)$ is also minimum phase.

\spm{need a reference.  I found something from the Center for Computer Research in Music and Acoustics at Stanford:  Minimum Phase (MP) polynomials in $ z$, \url{https://ccrma.stanford.edu/~jos/prf/}
I see he gives a proof without reference on page 4 of this doc \url{https://ccrma.stanford.edu/~jos/prf/prf.pdf}
\\
Ana found this:
Direct Adaptive Control Algorithms
Theory and Applications. Kaufman, Howard, Barkana, Itzhak, Sobel, Kenneth
\\
I think we should look here (our friend John!)
K. Zhou, J.C. Doyle, K. Glover, Robust and Optimal Control,
Prentice-Hall, Englewood Cliffs, NJ, 1995.
They are cited here \url{https://www.ee.iitb.ac.in/~belur/ee752/files/sorensen-passivity-SCL.pdf}
}

The remainder of the paper is organized into four sections.   Several design techniques are introduced in 
\Section{s:design},  and  \Section{s:LTI} presents general methodology for constructing and analyzing the linearized mean-field dynamics.    Examples are contained in
\Section{s:ex}, and conclusions in
\Section{s:conc}.

%

\section{Design} 
\label{s:design}

We first summarize the assumptions and notation.  

\subsection{Assumptions and Notation}   

A Markovian model for an individual load is created based on its typical operating behavior. 
\spm{axe for CDC. For example,  a water chiller turns on or off depending upon the temperature of the water.   }This is modeled by a Markov chain with transition matrix denoted $P_0$, with state space  $\state=\{x^1,\dots, x^d\}$;  it is assumed to be \textit{irreducible and aperiodic}.  
It follows that $P_0$ admits a unique invariant pmf (probability mass function), denoted $\pi_0$,  and satisfying $\pi_0(x)>0$ for each $x$.  

It is assumed throughout this paper that the family of transition matrices used for distributed control is of the form,
\begin{equation}
 P_\zeta(x,x')
\eqdef P_0(x,x')\exp\bigl(  \EF_\zeta(x,x')  -  \EFn{h_\zeta}(x)    \bigr)
\label{e:Pzeta}
\end{equation}
in which $\EF_\zeta$ is continuously differentiable in $\zeta$, and $\EFn{h_\zeta}$ is the normalizing constant
\begin{equation}
    \EFn{h_\zeta}(x)     
\eqdef  \log\Bigl( \sum_{x'} P_0(x,x')\exp\bigl(  \EF_\zeta(x,x')     \bigr) \Bigr)
\label{e:NormConstant}
\end{equation}
Each $P_\zeta$ must also be irreducible and aperiodic.

For any transition matrix $P$, an invariant pmf  is interpreted as a row vector, so that invariance can be expressed $\pi P=\pi$.   Any function $f\colon\state\to\Re$ is interpreted as a $d$-dimensional column vector,  and we use the standard notation $Pf\, (x) =   \sum_{x'}P(x,x')f(x')$,  $x\in\state$.

Several other matrices are defined based on $P$ and $\pi$: 
The \textit{adjoint} of $P$ (in $L_2(\pi)$)   is the transition matrix defined by
\begin{equation}
P^\adj (x,x') =  \frac{ \pi(x') }{ \pi(x)} P (x',x) ,\quad x,x'\in\state.
\label{e:adj}
\end{equation}
The \textit{fundamental matrix} is the inverse
\begin{equation}
Z = [I - P + 1\otimes \pi]^{-1}=\sum_{n=0}^\infty  [P - 1\otimes \pi]^n
\label{e:Z}
\end{equation}
with $ [P - 1\otimes \pi]^0 \eqdef I$  (the $d\times d$ identity matrix),
$1\otimes \pi$ is a matrix in which each row is identical, and equal to $ \pi$,
and $[P - 1\otimes \pi]^n = P^n -  1\otimes \pi$ for $n\ge 1$.  


The \textit{Donsker-Varadhan rate function} is denoted,
\begin{equation}
K(P\| P_0) = \sum_{x,x'} \pi(x) P(x,x')   \log \Bigl(\frac{P(x,x') }{P_0(x,x')} \Bigr)
\label{e:DVrate}
\end{equation}  
It is used here to model the cost of deviation from the nominal transition matrix $P_0$, as in   \cite{tod07,guaragwil14,meybarbusyueehr15,busmey14}.

\spm{axe for CDC.
It defines the relative entropy rate between two stationary Markov chains, and appears in the theory of large deviations for Markov chains cite{konmey05a}.     
}

\paragraph*{Nature \&\ nurture}   
In many applications it is necessary to include a model of randomness from nature along with the randomness introduced by the local control algorithm (nurture). \spm{shall we put in parens (nurture) ??}

Consider a load model in which the full state space  is the cartesian product of two finite state spaces:  $\state= \stateu\times\staten$,   where  $\stateu$ are components of the state that can be directly manipulated through control.   
The ``nature'' components $\staten$ are not subject to direct control.  For example, these components may be used to model the impact of the weather on the climate of a building.    


 Elements of   $\state$ are   denoted   $x=(x_u,x_n)$.
Any state  transition matrix under consideration is assumed to have the following conditional-independence structure,
\begin{equation}
P(x,x') = R(x, x_u') Q_0(x,x_n') ,
\label{e:PQ0R}
\end{equation}
for $ x\in\state, x_u'\in\stateu,\ x_n'\in\staten $,
where
$
\sum_{x_u'} R(x, x_u')=\sum_{x_n'}  Q_0(x,x_n') =1
$
for each $x$.
The matrix $Q_0$ is out of our control -- this models load dynamics and exogenous disturbances.
\spm{
, such as the  weather.   
}

\notes{connection with general MDPs will be in math/control  paper}
%
%
 
 \notes{\rDD\ for randomized demand dispatch?  We need an acronym.
 \\
 Ana suggests RDD.   Let's think about it.}
 
\subsection{Common structure for design}
\label{s:commonDD}

The construction of the family of functions $\{h_\zeta : \zeta\in\Re\}$   in \eqref{e:Pzeta}   is achieved using the following steps.

\wham{Step 1:}    
The specification of a function $\clH$ that takes as input a transition matrix $P$ that is irreducible.  The output
$H=\clH(P)$  is a real-valued function on the product space $\state\times \state$.  That is,  $H(x,x') \in   \Re$ for each pair $(x,x')\in \state\times \state$.  
\\
\wham{Step 2:}  
The family of transition matrices $\{P_\zeta\}$ and functions   $\{h_\zeta\}$ are defined by the solution to the $d$-dimensional ODE: 
\begin{equation}
\ddzeta h_\zeta = \clH( P_\zeta) ,\qquad \zeta\in\Re,
\label{e:hH}
\end{equation}
in which $P_\zeta$ is determined by $h_\zeta$ through \eqref{e:Pzeta}.  
The boundary condition for this ODE is $h_0 \equiv 0$.

In the special case in which randomness from nature is not considered,  we can apply the methods described here using  $\staten =\{ x_n^1\}$ (a singleton).  

\spm{axe: In this case,
the method will
generate a family of functions $\{h_\zeta\}$ that do not depend on   $(x_n, x_n')$.}
\spm{fixed following Ana's suggestion}

The conditional independence constraint \eqref{e:PQ0R} imposes constraints on the functions $\{h_\zeta\}$ and the transformation $\clH$.   To ensure that $P_\zeta$ is of the form \eqref{e:PQ0R}, it is sufficient to restrict to functions 
 $h_\zeta$ of $ (x,x')$ that do not depend on $x_n'$,  where $x'=(x_u',x_n')\in\state$.  
For this reason we make the notational convention, 
\[
h_\zeta(x,x') = h_\zeta(x, x_u'),\qquad x\in\state, \  x'=(x_u',x_n')\in\state.
\]
Since $h_\zeta$ is constructed through the ODE \eqref{e:hH},  we impose the same constraints on  $H=\clH(P)$: 
 \[
H(x,x') = H(x, x_u'),\qquad x\in\state, \  x'=(x_u',x_n')\in\state.
\]

Given any function $\preH\colon \state \to \Re$, the function defined below satisfies this constraint:
\spm{no room for this $\,, \quad x\in\state, \  x'=(x_u',x_n')\in\state.$}
\begin{equation}
H(x,x'_u) = \sum_{x_n'} Q_0(x,x_n') \preH(x_u',x_n')  
\label{e:Hnn}
\end{equation}
Each of the methods that follow construct  $H=\clH(P)$ of this form.  Hence the design problem reduces to choosing a mapping $\preH=\preclH(P)$.

\spm{axe: Each of these methods use variations  of the fundamental matrix,  and each relies on the function $\util$ that represents power consumption.  
}

The following normalization is imposed throughout:  The transition matrix $P_\zeta$ defined in \eqref{e:Pzeta} does not change if we add a constant to the function $\EF_\zeta$. We are thus free to normalize  $\preH=\preclH(P)$ by a constant.  Throughout the paper we fix  a state $\xz\in\state$,  and design $\preclH$ so that  $ \preH(\xz)=0$ for any $P$.

The ODE method can be simplified based on these observations. The proof of \Prop{t:preH}
 is straightforward.  
\begin{proposition}
\label{t:preH}
Consider a solution to the ODE \eqref{e:hH} in which $H=\clH(P)$ is of the form  \eqref{e:Hnn} for any matrix $P$.  It then follows that each of the functions $\{h_\zeta\}$ are of this form:
\[
h_\zeta(x,x'_u) = \sum_{x_n'} Q_0(x,x_n') h_\zeta^\circ(x_u',x_n') ,   
\]
$ x\in\state, \  x'=(x_u',x_n')\in\state$,
for some $h_\zeta^\circ\colon\state\to\Re$.  Moreover, these functions   solve the $d$-dimensional ODE,
\[
\ddzeta h_\zeta^\circ = \clH^\circ( P_\zeta) ,\qquad \zeta\in\Re,
\]
in which $P_\zeta$ is determined by $h_\zeta$ through \eqref{e:Pzeta},  and with boundary condition  $h_0^\circ \equiv 0$.
\bqed
\end{proposition}

\spm{{Important realization Feb 7, 2016:} $H$ is a function on $\state\times \state$,  while  $\preH$ is a function on $\state$.
Doesn't this mean we have made the ODE too complex?  ($d^2$ instead of $d$ dimensional).  Fixed with new proposition}

\subsection{Individual Perspective}
\label{s:IPD}

In this design,  the mapping $\preH=\preclH(P)$ is defined in terms of the fundamental matrix:

  \notes{Is IPD "solution" ok? why not "design"?}


\begin{DesignCDC}{\textbf{IPD solution:}}
Given   $P$,  the fundamental matrix $Z$ is obtained from \eqref{e:Z}, and then for each $x\in\state$,
\begin{equation}
\preH(x) = \sum_{x'} [ Z(x,x')-Z(\xz,x') ] \util (x').
\label{e:fishP}
\end{equation}
\end{DesignCDC}

The function $\preH$ specified in \eqref{e:fishP}
 is a solution to Poisson's equation,
\begin{equation}
P \preH =\preH -\util +\meanutil
\label{e:fish}
\end{equation}
where $\meanutil $ (also written $\pi(\util)$) is the steady-state mean:
\begin{equation}
 \meanutil  \eqdef \sum_x\pi(x) \util(x) 
\label{e:barutil}
\end{equation}
The function \eqref{e:fishP}
 is the unique solution satisfying $\preH(\xz)=0$ \cite[Thm.~17.7.2]{MT}.  

This choice for $\preH$ is called the Individual Perspective Design (IPD) since $h_\zeta$ solves an optimization problem formulated from the point of view of a single load.  
\spm{not clear: The optimization problem is an infinite-horizon optimal control problem of the following form:}
Given  $\zeta\in\Re$,  the ``optimal reward'' is defined by the maximum,
\begin{equation}
\eta^*_\zeta =\max_{\pi,P} \bigl \{\zeta \pi(\util) -  K(P\| P_0)  :  \pi P = \pi \bigr\}
\label{e:ARobjectiveReward}
\end{equation}
and $P$ is also subject to the structural constraint \eqref{e:PQ0R}.  The maximizer defines a transition matrix that is denoted,
\begin{equation}
\cP_\zeta =\argmax_{P} \bigl \{\zeta \pi(\util) -  K(P\| P_0)  : \pi P = \pi \bigr\}
\label{e:ARobjective}
\end{equation}
 
It is shown in \Theorem{t:IPD} that the optimal value $\eta^*_\zeta$ together with a \textit{relative value function} $h^*_\zeta$ solve the average reward optimization equation (AROE):
\begin{equation}
\max_P\Bigl\{ \clW_\zeta(x,P) 
+\sum_{x'} P(x,x') h^*_\zeta(x') \Bigr\} = h^*_\zeta(x) + \eta^*_\zeta
\label{e:AROE}
\end{equation} 
where $\clW_\zeta(x,P)  =\zeta \util(x) 
		-  \sum_{x'}  P(x,x') \log \bigl(\frac{P(x,x') }{P_0(x,x')} \bigr)  $.
		
The relative value function is not unique, since we can add a constant to obtain a new solution. We normalize this function so that  $h^*_\zeta(\xz)=0$.   The proof of the following can be found in the working paper \cite{busmey15a}.
 
\begin{theorem}
\label{t:IPD}
The IPD solution results in a collection of  transition matrices $\{P_\zeta: \zeta\in\Re\}$ with the following properties.   For each $\zeta$,
\begin{romannum}

\item  The transition matrix is optimal,
$P_\zeta = \cP_\zeta$.

\item  
For each $x$ and $x_u'$,  the function $h_\zeta$ that defines $P_\zeta$ is given by,
\begin{equation}
 h_\zeta( x, x'_u)  =  \sum_{x_n'} Q_0(x,x_n') h^*_\zeta(x_u',x_n')
\label{e:hmid}
\end{equation}
where $(h^*_\zeta, \eta^*_\zeta)$ solves the AROE \eqref{e:AROE}.

\item
The steady-state mean power consumption  satisfies,
\begin{equation}
\ddzeta \meanutil_\zeta =  \FRAC{1}{d^2}{d\zeta^2} \eta^*_\zeta \ge 0
\label{e:der-util-IPD}
\end{equation}
and   hence $\meanutil_\zeta$ is monotone in $\zeta$.  \spm{Keep this here and prove!!}\bqed
\end{romannum}
\end{theorem}


\subsection{System Perspective}

The motivation for the following System Perspective Design (SPD) is from the point of view of the BA.
Under general conditions,  the linearized aggregate model is passive,
which is a desirable property from the grid-level perspective. 

The construction of $\preH=\preclH(P)$ is similar to  IPD.  For any matrix $P$
with invariant pmf $\pi$, recall the definition of the adjoint  $ P^\adj $ in \eqref{e:adj}.
The matrix product is denoted
\[
P^\adjnat (x,x') =  \sum_{z\in\state} P^\adj (x,z) P (z,x') ,\qquad 
 x,x'\in\state .
\] 
The fundamental matrix  defined in terms of this transition matrix is denoted $Z^\adjnat = [I - P^\adjnat + 1\otimes \pi]^{-1}$.
\begin{DesignCDC}{\textbf{SPD solution:}}
Given   $P$,  the  matrix $Z^\adjnat$ is obtained, and
\begin{equation}
\!\!
\preH(x) = \sum_{x'} [ Z^\adjnat(x,x')-Z^\adjnat(\xz,x') ] \util (x')\, ,\ \ x\in\state.
\label{e:fishW}
\end{equation}
\end{DesignCDC}
Under additional assumptions, the algorithm obtained from SPD results in a positive real linearization.
The proof of the bound \eqref{e:G+bdd}  is contained in 
\Section{s:passive}. 

\begin{theorem}
\label{t:SPD}
Suppose that the Markov chain with transition matrix $P^\adjnat_0 = P_0^\adj P_0$ is irreducible,
and that $P_0=R_0$  (a model without probabilistic constraints).


Then, the solution to the SPD satisfies the following strict positive-real condition:
the  linearized model at any constant value $\zeta$ obeys the bound,
\notes{define linearized model!  The definition appears in the next section.}
\begin{equation}
G^+_\zeta(e^{j\theta}) +G^+_\zeta(e^{-j\theta})  \ge   \sigma^2_\zeta,\qquad \theta\in\Re
\label{e:G+bdd}
\end{equation}
where $\sigma^2_\zeta$ is the variance of $\util$ under $\pi_\zeta$.
\bqed
\end{theorem}

The irreducibility assumption on $P^\adjnat_0  $ does not come for free.  Consider for example the  Markov chain on $d\ge 3$ states defined by  $P_0(x^i,x^{i+1}) = 1$ for $1\le i\le d-1$,  and $P_0(x^d,x^d) =P_0(x^d,x^1) = 1/2$.    This chain is irreducible and aperiodic. The behavior of the adjoint is similar;  in particular, $P_0^\adj(x^{i+1},x^i) = 1$ for each $1\le i\le d-1$.   It follows that   $P^\adjnat_0 (x^k,x^k) =1$ for $2\le k\le d$,  so the irreducibility assumption fails.

\subsection{Exponential family}

Rather than solve an ODE,  it is natural to fix a function 
$\preHe\colon\state\to\Re$, and define  for each $x,x'_u$ and $\zeta$,  
\[
\begin{aligned}
h_\zeta(x,x'_u) &=  \zeta\He(x_u' \mid x) 
\\[.2cm]
 \text{with} \quad\He(x_u' \mid x) &\eqdef  \sum_{x_n'} Q_0(x,x_n') \preHe(x_u',x_n') 
\end{aligned}
\]
This  is a special case of the two-step design described in \Section{s:commonDD} in which $\preclH(P)  =   \preHe  $, independent of $P$,  and
the function $H=\clH(P)$ is then obtained from \eqref{e:Hnn}.

In this case, the transition matrices defined in \eqref{e:Pzeta} can be regarded as an \textit{exponential family}.   The exponential family using $\preHe = \util$ will be called the  \textit{myopic design}. 

Other designs can be obtained as linear approximations to the IPD or SPD solutions, with  $\preHe= \preclH(P_0)$.  In the linear approximation of the IPD solution, this is a solution Poisson's equation for the nominal model:  
 \begin{equation}
  P_0 \preHe = \preHe - \tilutil_0
\label{e:fishe}
\end{equation} 
where  $\tilutil_0(x)=\util(x) - \meanutil_0$.  The resulting exponential family is called the IPD${}_0$ design.  It is
 approximately optimal for $\zeta$ near zero -- a proof of \Theorem{t:OptExpMain}
 can be found in \cite{busmey15a}.

\spm{reference to proof added}

\begin{theorem}
\label{t:OptExpMain}
The following approximations hold for the  transition matrices $\{P_\zeta\}$ obtained from the IPD${}_0$ design:
\begin{romannum}
\item  With $\cP_\zeta$ the optimal transition matrix in \eqref{e:ARobjective},

\begin{equation}
P_\zeta(x,x') = \cP_\zeta(x,x') +O(\zeta^2),\quad \text{for all  $x,x', \zeta$}
\label{e:PzetaApproxP}
\end{equation}
\item  Let $\eta_\zeta=\zeta \pi_\zeta(\util) -  K(P_\zeta\| P_0) $ denote the value 
of the quantity in brackets in \eqref{e:ARobjectiveReward}
that is obtained using $(\pi_\zeta,P_\zeta)$.  Then,
$
\eta_\zeta  = \eta_\zeta^* +O(\zeta^4)   
$.
\bqed
\spm{commented out complex "consequently"}
\end{romannum}
\end{theorem}

A similar result holds if $\preHe $ is chosen based on SPD, with
\[
\preHe = [I - P^\adjnat_0 + 1\otimes \pi_0]^{-1}
\]
The linearization at $\zeta=0$ will be positive-real   under the assumptions of 
\Theorem{t:SPD},  because \eqref{e:G+bdd} continues to hold at $\zeta=0$,
\[
G^+_0(e^{j\theta}) +G^+_0(e^{-j\theta})  \ge   \sigma^2_0,\qquad \theta\in\Re.
\]
An example in \Section{s:tcl} shows that passivity may fail for 
the linearization $G^+_\zeta$ at values of $\zeta$ far from zero.
\spm{we may remove this example}

\subsection{Geometric sampling}

\spm{Ana points out that this is a crappy lead-in.  I'll try on Monday, or perhaps AB can have a go at
substitute text}

Geometric sampling is specified by a transition matrix $S_0$ and a fixed parameter $\gamma\in(0,1)$.  At each time $t$, a weighted coin is flipped with probability of heads equal to $\gamma$.  If the outcome is a tail,  then the   state does not change.  Otherwise,  a transition is made from the current state $x$ to a new state $x'$ with probability $ S_0(x,x')$.   The overall transition matrix is expressed as a convex combination, 
\begin{equation}
P_0 = (1-\gamma)I + \gamma  S_0
\label{e:geoP}
\end{equation}
One motivation for sampling in \cite{chebusmey15b} is to reduce the chance of excessive cycling at the loads,   while ensuring that the data rate from balancing authority to loads is not limited.   It was also found that  this architecture justified a smaller state space for the Markov model.

Based on this nominal model, there are two approaches to applying the  design techniques introduced in this paper.  If $P_0$ is transformed directly,  then the resulting family of transition matrix will be of the form, 
\begin{equation} 
P_\zeta = (1-\gamma_\zeta)I + \gamma_\zeta  S_\zeta
\label{e:geoPzTotal}
\end{equation}
in which $\gamma_\zeta$ is a function of $x$.  That is, if at time $t$ the state is $X(t)=x$ and the input $\zeta_t=\zeta$, then once again a weighted coin is flipped, but with probability of success equal to $\gamma_\zeta (x)$.   Conditioned on success,  a transition is made to state $x'$ with probability $ S_\zeta(x,x')$.  

\spm{Ana asks if \eqref{e:geoPz} would arise if   we declare that the coin flips are from nature.
This is an open question:  If we apply our IPD or SPD to $P_0$, do we get the same family of $S_\zeta$
as when we apply the same algorithm to $P_0$?   }

In some cases it is convenient to fix the statistics of the sampling process, and transform $S_0$ using any of the design techniques described in the previous subsections.   Once the family of transition matrices $\{S_\zeta :\zeta\in\Re\}$ is constructed, we then define 
\begin{equation}
P_\zeta = (1-\gamma)I + \gamma  S_\zeta
\label{e:geoPz}
\end{equation}
Each approach is illustrated through examples in \Section{s:ex}.

\spm{I am trying to explain a little subtlety here.  The coin flip -- is this nature, or nurture?   If it is nature, then we have to use our more complex design.   If it is nurture, then $P_\zeta$ will have a coin flip that depends on state.  This would be cool!   In introducing \eqref{e:geoPz} I am proposing a new idea in which we treat $S_0$ as the input to any of our algorithms.  
 }

\section{Linearized Mean-Field Model} 
\label{s:LTI}

In this section we describe structure for the linearized model in full generality.  We consider a general family of transition matrices of the form \eqref{e:Pzeta}, maintaining the assumption that $\EF_\zeta$ is continuously differentiable in $\zeta$, and that $P_0$ is irreducible and aperiodic.
 
 \subsection{Transfer function}

Representations of the transfer function for the linearization require a bit more notation.  
We denote  $\tilutil_\zeta = \util - \meanutil_\zeta$, with  $\meanutil_\zeta = \pi_\zeta(\util)$.
The derivative of the transition matrix is also a $d\times d$ matrix, denoted
\begin{equation}
\clE_\zeta=\frac{d}{d\zeta} P_\zeta  
\label{e:Pder}
\end{equation}
A simple representation for this matrix is obtained in \Prop{t:lin}, in terms of the function, 
\begin{equation}
H_\zeta(x,x') = \ddzeta \EF_\zeta\, (x,x') ,\quad x,x'\in\state.  
\label{e:hder}
\end{equation}

\spm{axe!!  The bullet comes up for something outside of this paper!}

The invariant pmf  $\pi_\zeta$ for $P_\zeta$ is regarded as the equilibrium state for the mean-field model   \eqref{e:MFJ}, with respect to the constant input value $\zeta_t\equiv \zeta$.   The linearization about this equilibrium is described in \Prop{t:lin}. The proof is omitted since it is minor generalization of    \cite[Prop.~2.4]{meybarbusyueehr15}.
\begin{proposition}
\label{t:lin} 
The linearization of  \eqref{e:MFJ} at a particular value $\zeta$ is the state space model with transfer function,
\begin{equation}
\TF_\zeta(z) = C [Iz-A]^{-1} B
\label{e:TF}
\end{equation}
in which
 $A=P_\zeta^\transpose$,  $C_i=\tilutil_\zeta(x^i)$ for each $i$, and
\begin{equation}
B_i = \sum_x \pi_\zeta(x)\clE_\zeta(x,x^i),\qquad 1\le i\le d
\label{e:Bzeta}
\end{equation}
\bqed
\end{proposition}

\spm{missing bullets added to prop.}

Another representation of $B$ is obtained based on  the product $P^\adjnat_\zeta=P^\adj_\zeta P_\zeta$,
where $P^\adj_\zeta$ denotes the adjoint of $P_\zeta$.
\begin{proposition}
\label{t:B}
The derivative of the transition matrix can be expressed in terms of the function \eqref{e:hder}:
\begin{equation} 
\frac{\clE_\zeta(x,x') }{P_\zeta(x,x') } =H_\zeta(x,x') - P_\zeta H_\zeta\, (x)
\label{e:clElogDer}
\end{equation}
where
$
P_\zeta H_\zeta \, (x) = \sum_{x'} P_\zeta (x,x') H_\zeta (x,x')$ for $ x\in\state$.
In the special case in which $H_\zeta (x,x')$ is independent of $x$, the entries of the vector $B$ can be expressed, 
\begin{equation} 
B_i 
=  \pi_\zeta (x^i)[ H_\zeta (x^i) - P^\adjnat_\zeta  H_\zeta \, (x^i) ]
\label{e:BadjForm}
\end{equation}
\end{proposition}

\begin{proof}
For each $x,x'$,
\[
\ddzeta \log(P_\zeta(x,x')) =   H_\zeta(x,x') - \ddzeta \EFn{h_\zeta}(x)
\]
where we have used the definition $  H_\zeta = \ddzeta h_\zeta $.
The derivative is computed using  \eqref{e:NormConstant}, giving
$ \ddzeta \EFn{h_\zeta}\, (x) = {}$
\[
\begin{aligned}
 &= \ddzeta \log \Bigl( \sum_{x'} P_0(x,x')\exp\bigl(  \EF_\zeta(x,x')     \bigr) \Bigr)
 \\
  &=\exp(- \EFn{h_\zeta}(x))   \sum_{x'} P_0(x,x')\exp\bigl(  \EF_\zeta(x,x')     \bigr) H_\zeta(x,x')  
  \\
    &=    \sum_{x'} P_\zeta (x,x')  H_\zeta(x,x')  
\end{aligned}
\]
which implies \eqref{e:clElogDer}.
If $H_\zeta$ depends only on $x'$ then,
\[
B_i = \sum_x \pi_\zeta(x) P_\zeta(x, x^i)  [H_\zeta(x^i) - P_\zeta H_\zeta\, (x) ]
\]
We can write $ \pi_\zeta(x) P_\zeta (x, x^i) = \pi_\zeta(x^i) P^\adj_\zeta(x^i,x)$ to obtain
\[
\begin{aligned}
B_i  &=  \sum_x \pi_\zeta (x^i) P^\adj(x^i,x)  [H_\zeta (x^i) - P_\zeta H_\zeta \, (x) ]  
\\
& =\pi_\zeta (x^i)H_\zeta (x^i)  -  \pi_\zeta (x^i) \sum_x  P^\adj(x^i,x)  [ P_\zeta H_\zeta \, (x) ]
\end{aligned}
\]
\end{proof}

\subsection{Power spectral density and the positive real condition}
\label{s:passive}

In \cite{busmey14} the transfer function \eqref{e:G+} was considered for a linearized mean-field model in continuous time.   A representation of this transfer function used in this prior work admits a counterpart in the discrete-time setting.

The infinite series on the right hand side of \eqref{e:G+} suggests that we require a probabilistic interpretation of the scalar $ C A^k B  $,  where $(A,B,C) $ are given in \Prop{t:lin}.  This is achieved on defining $\tilutilB(x^i) = B_i/\pi_\zeta(x^i)$ for each $i$;  this is a function on $\state$ whose mean is zero:  $0=\sum_x \pi_\zeta(x) \tilutilB(x)$.  
\begin{lemma}
\label{t:CAkB}
Let $\bfmX$ denote a stationary realization of the Markov chain with transition matrix $P_\zeta$, so that in particular,  $X(k)\sim \pi_\zeta$ for each $k$.  Then, 
\begin{equation}
C A^k B = \Expect[ \tilutilB(X(0) )  \tilutil_\zeta(X(k) )  ] 
\label{e:CAkB}
\end{equation}
\end{lemma}

\begin{proof}
We have by definition, $C A^k B = B^\transpose P_\zeta^k C^\transpose$, which can be expressed as the sum,
\[ 
B^\transpose P_\zeta^k C^\transpose =
\sum_{i, j}  \pi_\zeta(x^i) \tilutilB(x^i)   P^k_\zeta(x^i,x^j)\tilutil_\zeta (x^j)   
\]
This is equivalent to \eqref{e:CAkB}.
\end{proof} 

With these identities in place we are ready to prove the passivity bound in \Theorem{t:SPD}.


\paragraph*{Proof of \Theorem{t:SPD}}
In the SPD solution without probabilistic constraints, it follows from the design rule \eqref{e:fishW}
 and the representation for the vector $B$ in  \eqref{e:BadjForm} that $
 \tilutilB = \tilutil_\zeta$.  \Lemma{t:CAkB}
 gives the covariance interpretation,
\[
C A^k B = \Expect[\tilutil(X(0)) \tilutil(X(k)) ] 
\]
Let $S^+(\theta)$ denote the power spectral density,
\[
S^+(\theta) = \sigma^2_\zeta+\sum_{k=1}^\infty \Expect[\tilutil(X(0)) \tilutil(X(k)) ] [e^{jk\theta} + e^{-jk\theta} ]
\]
where $\sigma^2_\zeta =\Expect[\tilutil(X(0))^2 ]$.  The bound thus follows from the definitions:
\[
G^+_\zeta(e^{j\theta}) +G^+_\zeta(e^{-j\theta})  = S^+(\theta) + \sigma^2_\zeta\ge \sigma^2_\zeta
\]\vspace{-2em}
\bqed

\spm{conclude somewhere with a statement on the Kalman-Yakubovich-Popov (KYP) Lemma lemma}

\section{Examples} 
\label{s:ex}

\subsection{Rational pools}
\label{s:pp}

\spm{
The pool examples are done -- see v2 notes -- we just need to clean this up.   Biggest problem in this section is that we are suddenly transforming $S_0$ instead of $P_0$.   I think we need a final subsection in Section III entitled ``random sampling'' with motivation and results.
\\
Note that it would also be cool to transform $P_0$ even with geometric sampling.  The sampling rate would then be state dependent.  
\\
We might start from scratch on the TCLs.  Joel's model is very good.}

The load in this case is a pool pump used to maintain water quality in a residential pool.  The pump is assumed to consume 1~kW of power when operating.  In the nominal model it is assumed that it runs for between 8 and 14 hours per day.
In the original model of \cite{meybarbusyueehr15}, the
 state space was taken to be the finite set,
\begin{equation}
\state=\{  (m,k) :  m\in  \{  \oplus,\ominus\}  ,\  k \in \{1,\dots,\clI\} \}  
\label{e:poolstate}
\end{equation}
where $\clI>1$ is an integer.   For the $i$th load,  if $X(t)^i = (\oplus, k)$,  this means that the pool pump is on at time $t$, and has remained on for the past $k$ time units.  

\begin{figure}[h]
\Ebox{.5}{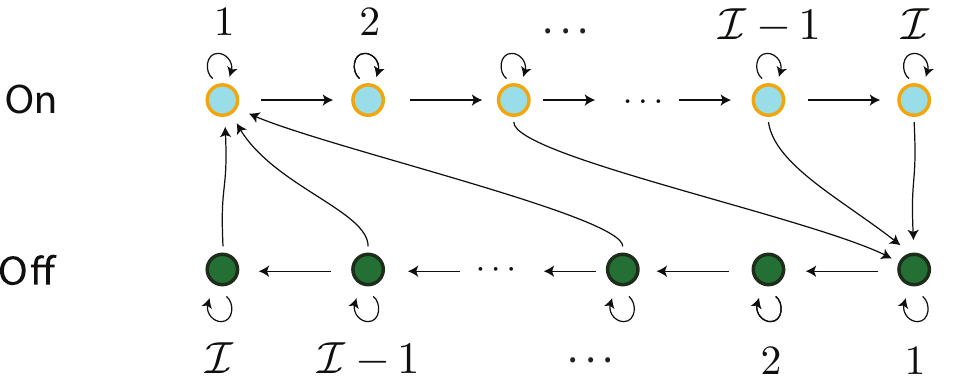} 

\caption{State transition diagram for pool pump model.}
\label{f:ppp}
\vspace{-1ex}
\end{figure} 

In this paper we take the same state space,  but with a different interpretation of each state:
Here as in \cite{chebusmey15b} we employ geometric sampling, so that the nominal state transition matrix is of the form \eqref{e:geoP}. The state transition diagram for $P_0$ is shown in Fig{f:ppp}.

In the experiments that follow, the transition matrix $S_0$ is the model with 12-hour cleaning cycle from  \cite{meybarbusyueehr15}, in which $\clI=48$, and hence $d=|\state|=96$.  It is assumed that the BA sends a signal every five minutes, and that the geometric sampling parameter is $\gamma=1/6$.  Consequently, the state of each load changes every 30 minutes on average. 

\begin{figure}[h]
\Ebox{.55}{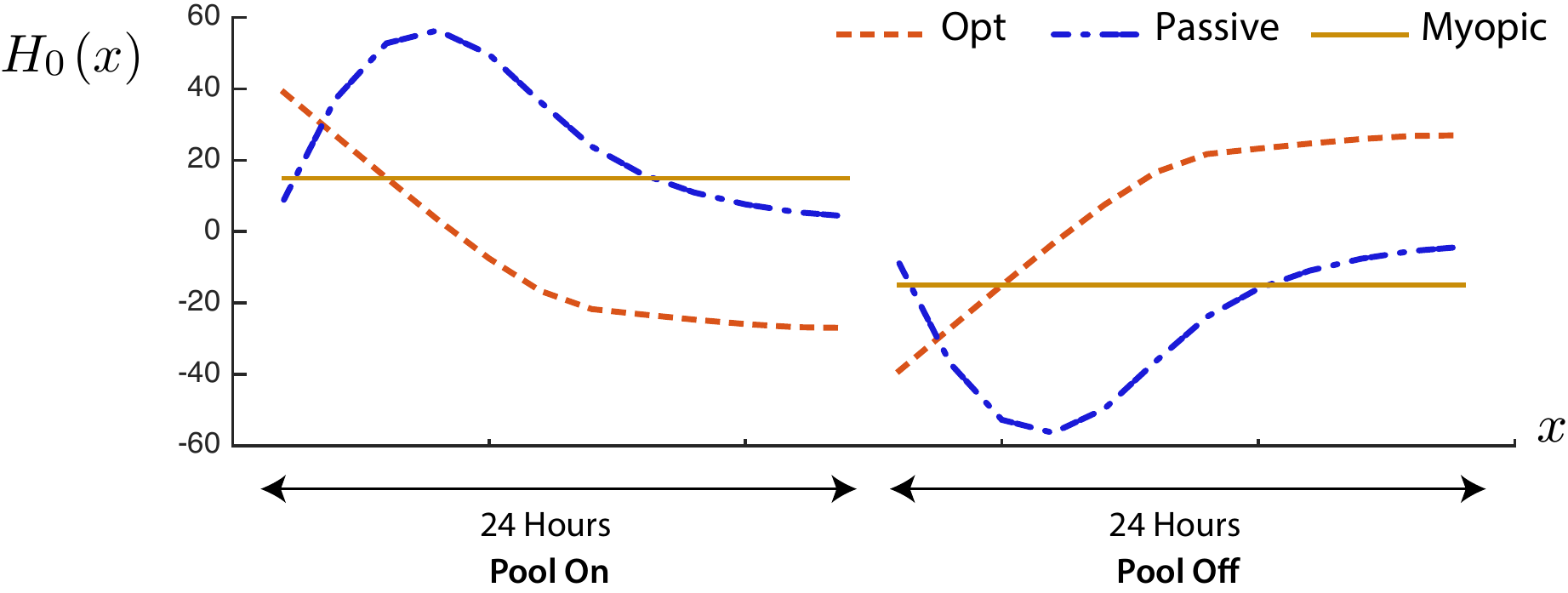}

\caption{$H_0=\clH(P_0)$ in  three designs for the pool pump model.}
\label{f:Pool3H} 
\vspace{-1.5em}
\end{figure}

There is no need to model uncertainty from nature.  Hence the function $H=
\clH(P)$ will depend only on its second variable:  $H(x,x') = H(x')$ for all $(x,x')$.

In each of these experiments the nominal transition matrix $P_0$ was taken as an input to the algorithm, and not $S_0$.  The resulting transition matrix $P_\zeta$ is of the form given in \eqref{e:geoPzTotal},  in which the sampling rate is state-dependent for non-zero $\zeta$.   For the SPD solution it was necessary to use $P_0$ as the input: It can be shown that the transition matrix $P^\adjnat_0 $ is irreducible,  but  $S^\adjnat_0 \eqdef 
S_0^\adj S_0$ is \textit{not} irreducible in this example.   Recall from \Theorem{t:SPD}
 that irreducibility is required to ensure the existence of the SPD solution.
\spm{Note!!}

\Fig{f:Pool3H} shows the function $H_0=\clH(P_0)$ obtained in three sets of experiments.   The normalization in the myopic design uses $H_0=30 (\util-1/2)$.  This is equivalent to using $\util$ since adding a constant does not impact $R_\zeta$,  and the multiplication by $30$ only scales $\zeta$.   A comparison of the three transfer functions $\TF^+$ obtained through a linearization at $\zeta=0$ is shown in \Fig{f:BodePools}.  The myopic design appears to be preferable to the two others:  in particular, the phase plot stays nearest to zero in this design.

\begin{figure}[h]
\Ebox{.475}{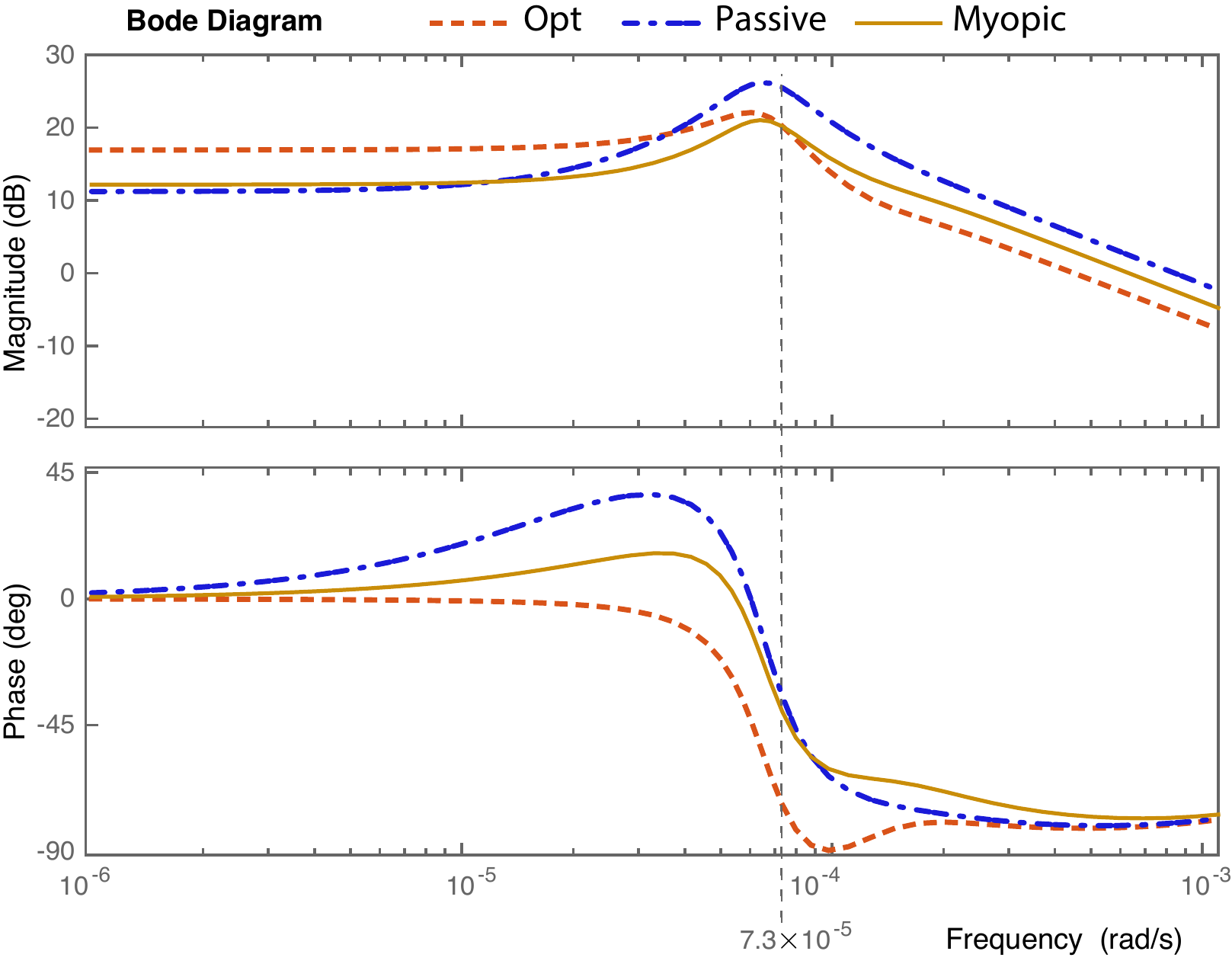}

\caption{Bode plots obtained from three designs for the pool pump model.}
\label{f:BodePools}
\vspace{-1.25ex}
\end{figure} 

One drawback with the myopic design is that we have no basis for analysis.  Difficulties are also observed in numerical experiments if we allow $|\zeta|$ to take on values far from zero.  Of the three designs, the IPD is found to be the most numerically stable in all of the experiments considered, in the sense that the dynamics change predictably with $\zeta$,  and the linearization about each value of $\zeta$ are nearly the same for a large range of $\zeta$.  

In the next subsection we move to a different class of loads in which we cannot ignore exogenous randomness.  We consider in greater depth the difference between the myopic and IPD outcomes for a wider range of $\zeta$.  It is found again that the two designs are very similar locally (for $\zeta\sim 0$),  but the myopic design is numerically unstable for $\zeta$ outside of a small neighborhood of the origin.

 \spm{Ana asks what happens if we apply myopic to $S_0$ instead of $P_0$}

\def\Thetaset{\text{$\Theta_{\text{set}}$}}

\def\Ptrans{\text{$P_{\text{trans}}$}}
\def\Tmin{\text{$\Theta_{\text{min}}$}}
\def\Tmax{\text{$\Theta_{\text{max}}$}}

\subsection{Thermostatically controlled loads}
\label{s:tcl}

A thermostatically controlled load (TCL) such as a water heater, refrigerator, or air-conditioner is a device for which temperature control is achieved using a dead-band.  It is assumed here that  
power consumption takes just two values (on or off).    To simplify discussion,   attention is directed to  a cooling device, such as a residential  refrigerator.  

We begin with a noise free model, described as the controlled linear system,
\begin{equation}
\Theta(k + 1) = \Theta(k) + (1 - \varrho)(\Theta_a - \Theta(k)  - m(k)\Theta_g),
\label{e:TCL}
\end{equation} 
where $m(k)=1$ or $0$ indicates if the unit is on or off, $\Theta_a$ is ambient temperature,  $\varrho$ depends on the dynamics of the load,  and $\Theta_g$ the depends on the physics of the device. One example  considered in \cite{johThesis12} is an air conditioner for which  $\Theta_g = R \Ptrans  = 2\times 14 = 28$ (this and other parameters are summarized in \Table{t:JMTable4.1}). 

A signal from the BA is broadcast at 20 second intervals.   
The parameter $\varrho$ is obtained based on this sampling time, and the product of thermal resistance and capacitance:      $\varrho = \exp[-20/(RC)]$,   with  $RC$  also in units of seconds. 
The value $RC=4$ obtained from \Table{t:JMTable4.1} is in units of hours --- on scaling to seconds we obtain,  
\[
\varrho = e^{-h/(RC)}=e^{-20/(4\times 60^2)} = e^{-1/720} \approx 1-1/720
\]

\begin{table}[h]
\centering
 \begin{tabular}{clc} 
Parameter \qquad & Meaning & Value
 \\
$\Thetaset$ &  set temperature set-point & $20^\circ$C
 \\
$[\Tmin ,\Tmax ]$  & temperature dead-band   & $20 \pm 0.5{}^\circ$C
 \\
$\Theta_a$ & ambient temperature & 32${}^\circ$C
 \\
$R$ & thermal resistance & 2${}^\circ$C/kW
 \\
$C$ & thermal capacitance & 2 kWh/${}^\circ$C    
\\
\Ptrans &  energy transfer rate & 14 kW
\end{tabular}
\caption{Homogeneous air conditioner parameters ---  mean data   from Table 4.1 of \cite{johThesis12}.}
\label{t:JMTable4.1}
\end{table}

This model is based on the physics of heating and cooling, but the  dynamics are accurately captured by a constant drift model:
\begin{equation}
\Theta(k + 1) = \Theta(k)  - m(k) \delta_- +(1-m(k)) \delta_+
\label{e:TCLdrift}
\end{equation}
With drift parameters $ \delta_\pm$ carefully chosen,  the behavior of the two models is barely distinguishable.

\spm{ as illustrated in Fig{f:LinearModelComparison}.}

%


The deterministic model \eqref{e:TCLdrift} is the basis of a stochastic model,
 \begin{equation}
\Theta(k + 1) = \Theta(k)  - m(k) \delta_- +(1-m(k)) \delta_+ +\Delta(k+1)
\label{e:TCLdriftP0}
\end{equation}
in which $\bfDelta$ is a zero-mean, i.i.d.\  sequence.  In the experiments that follow the sampling time was taken to be $2$~seconds,  and $\bfDelta$ was taken to be Gaussian with variance $10^{-6}$  (the small variance is justified with this fast sampling rate).

  \begin{figure}[h] 
\Ebox{.5}{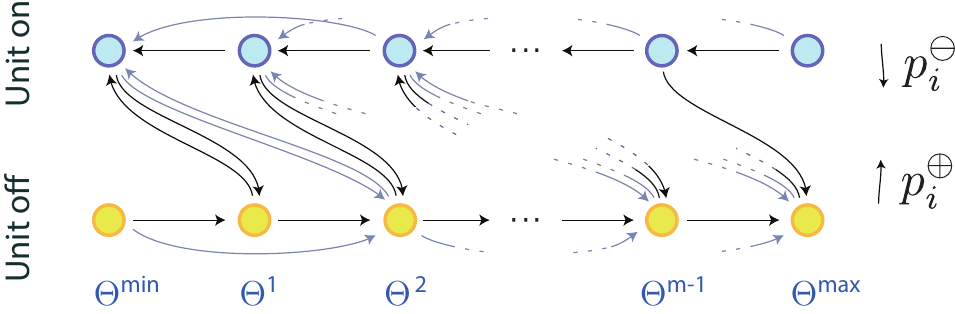}

\caption{State transition diagram for a TCL providing cooling. }
\label{f:fridgeMarkovDynamics}
\vspace{-1.25ex}
\end{figure} 

\spm{Self loops missing.  I agree -- self-loops exist.
I was concerned about complexity.
\\
I can add them, but then I fear I need vertical errors which would be hard to squeeze in.}

A Markov chain model can be constructed with state $X(k) = (m(k), X_n(k))$ evolving on $\state= \stateu\times\staten$,  where $X_n(k) = \Theta(k) $ and $\stateu=\{\ominus,\oplus\}$;  the  interpretation is the same as in the pool filtration model, with the interpretation $m(k) =\oplus$  is the same as $m(k)=1$ in \eqref{e:TCLdriftP0}.
\notes{annoying, I know}  
Temperatures are restricted to a lattice to obtain a finite state-space Markov chain. 
To obtain $d$ states, assume that $d\ge 4$ is an even number,  and discretize the interval $[\Tmin,\Tmax]$ into $d/2$ values as follows:   $\staten=\{  \Tmin + k T_\Delta:  0\le k\le d/2-1 \}$, in which the increments in the lattice are   $T_\Delta=  (\Tmax-\Tmin)/(d/2-1)$.

A nominal randomized policy for $\bfmm$ defines the transition matrix $R_0$.  Following the notation of   \cite{meybarbusyueehr15},  the nominal transition matrix for $\bfmm$ is defined by
\begin{equation}
\begin{aligned}
R_0(x,\oplus)& = 
\begin{cases}
1-p^\ominus(x_n) &   x= (\oplus, x_n)
\\
p^\oplus(x_n) &   x= (\ominus, x_n)
\end{cases}
\\[.2cm]
R_0(x,\ominus) &= 
\begin{cases}
 p^\ominus(x_n) &   x= (\oplus, x_n)
\\
1-p^\oplus(x_n) &   x= (\ominus, x_n)
\end{cases}
\end{aligned}
\label{e:R0pools}
\end{equation}


%
%

 As in \cite{meybarbusyueehr15}, the definition of $p^\ominus$ is based on the specification of a  cumulative distribution function $F^\ominus$ defined on the interval $[\Tmin,\Tmax]$.   This CDF is meant to model the statistics of the time interval during which the unit is off, for the model with continuous state space.   We define $p^\ominus(x_n) =1$ for $x_n=\Tmax$,  $p^\ominus(x_n) = F^\ominus(x_n)$ for $x_n=\Tmin$, 
and for all other values,
\[
p^\ominus(x_n) = [F^\ominus(x_n) - F^\ominus(x_n-T_\Delta)]/[1-F^\ominus(x_n-T_\Delta)]
\] 

In the experiments that follow, the general form taken for $F^\ominus$ was chosen in the parameterized family,
\[
F^\ominus(x_n) = \exp(-(\Tmax-x_n)^\rho/(2\sigma^\rho)), \  \Tmin\le x_n \le \Tmax ,
\]
with $\sigma,\rho>0$. The values $T_\Delta = 0.05$,  $\sigma=0.02$ and $\rho=0.75$ were used for $p^\ominus$ and $p^\oplus$ in the experiments surveyed here.
\spm{need to say something about oplus!}

%

 \begin{figure}[h]
\Ebox{.65}{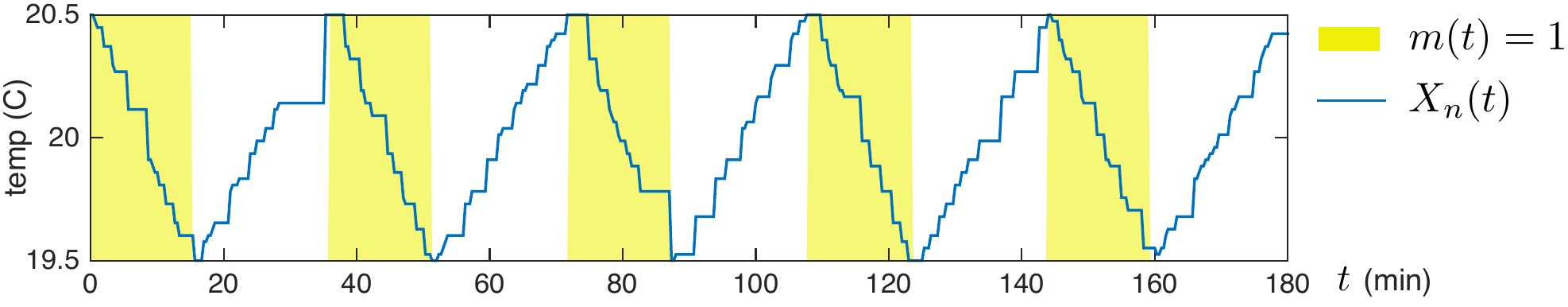}
 
\caption{Trajectory of TCL model with geometric sampling $\gamma=1/3$.}
\label{f:TCLtraj}
\vspace{-1.25ex}
\end{figure} 

In addition, geometric sampling was applied:  a  family of models of the form \eqref{e:geoPz} was constructed, in which $R_0$ was used to define $S_0$ in  \eqref{e:geoP}.   The construction of a model of this form requires a different interpretation of the nature component of the state  $\bfmX_n$.

To obtain dynamics of the form \eqref{e:geoP},  let $\{\tau_k\}$ denote the discrete renewal process in which
$\tau_0=0$  and $\{ \Delta_k = \tau_k - \tau_{k-1} : k \geq 1\}$ is i.i.d., with a geometric marginal:
\[
\Prob\{\Delta_k> n\} =\gamma^n,\qquad n\ge 0,  \ k\ge 1.
\]
The nature component of the state is constant on each discrete-time interval $(t: \tau_k\le t < \tau_{k+1})$,  with $X_n(t) = \Theta(\tau_k)$ on this interval.    Given a nominal randomized policy for the input process $\bfmX_u=\bfmm$,  the nominal transition matrix $Q_0$ can be estimated via Monte-Carlo based on a simulation of \eqref{e:TCLdriftP0},  or from measurements of an actual TCL.  

\spm{axe:
Once $Q_0$ is specified,  the matrix $S_0$ in \eqref{e:geoP} has entries $S_0(x,x') = R_0(x, x_u') Q_0(x,x_n')$.}

\spm{Figure inserted for our education.  It doesn't belong in CDC paper, so commented out for now}


\spm{axe for CDC:  }
%

\spm{I did verify that $S^\adjnat_0 = S_0^\adj S_0$ is irreducible in this example.   Its spectral gap is even a bit better: $\lambda_2(S^\adjnat_0 ) =  
    0.9813$  and $|\lambda_2(S^0 )| =     0.9885$.   Note that
\eqref{e:G+bdd} is not guaranteed since the key assumption of \Theorem{t:SPD} fails:  that $S_0$ is not subject to "probabilistic constraints"  (need a better term?).  }

A Markov model with this transition matrix would also require $m(t)$ constant on each of the intervals
$(t: \tau_k\le t < \tau_{k+1})$,  $k\ge 0$. In  simulations this constraint was  violated occasionally since $m(t)=1$ when $\Theta(t)>\Tmax$,  and $m(t)=0$ when $\Theta(t)<\Tmin$.   This leads to   modeling error that is small, provided $\gamma$ is not too close to unity.   \Fig{f:TCLtraj} shows an example of the evolution of $\bfmX_n$ with $\gamma=1/3$.   The temperature never violates the dead-band constraint because of the constraints imposed on $\bfmX_u$.

\spm{removed trajectories for CDC}


\spm{axe for CDC:}

\begin{figure*}[h!]
    \centering
\begin{subfigure}[t]{0.45\textwidth}
        \includegraphics[width=\hsize]{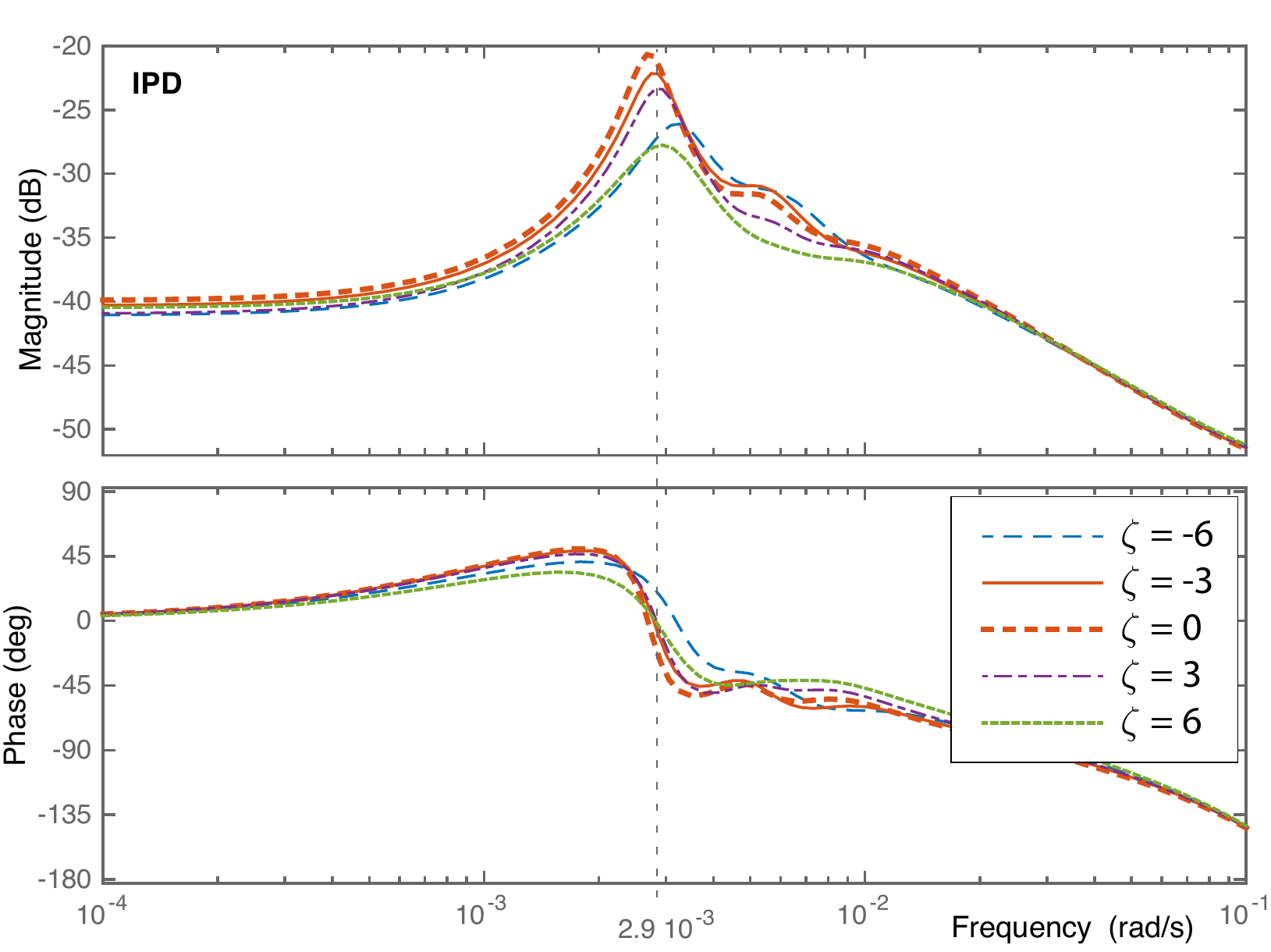}
 \centering
\caption{IPD Design}
\label{f:TCL-OptBode} 
\end{subfigure}  
\hfil
\begin{subfigure}[t]{0.45\textwidth}
\centering
        \includegraphics[width=\hsize]{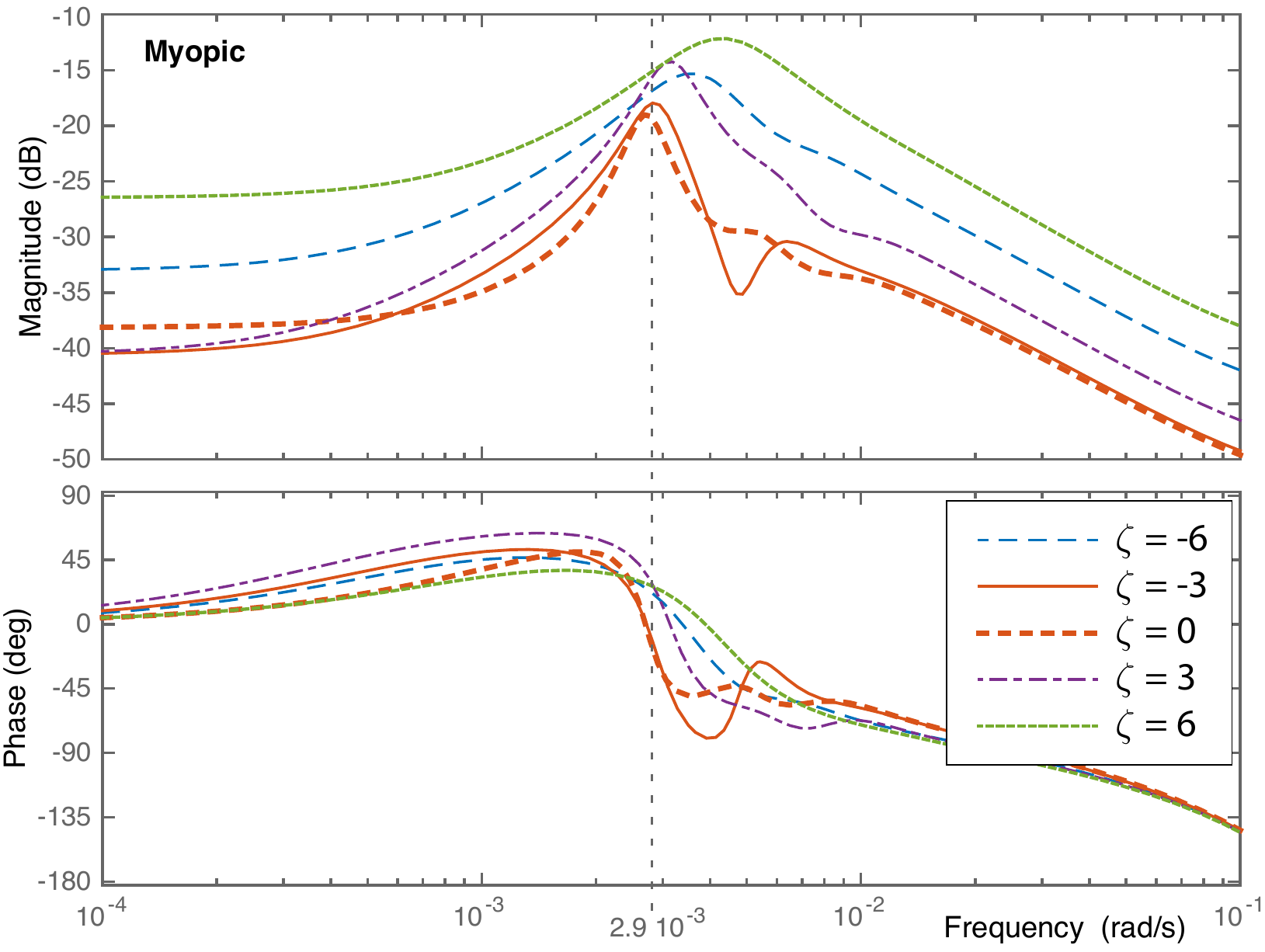}
 
\caption{Myopic Design}
\label{f:TCL-MyopicBode} 
\end{subfigure}  
\smallskip

\caption{Bode plots for two designs, based on linearizations at five values of $\zeta$.  Bode plots for the myopic design are  less reliable for $|\zeta| >3$. 
}
 \end{figure*}  
In this example the transition matrix $S^\adjnat_0 \eqdef S_0^\adj S_0$ is irreducible, so that the SPD solution is computable.   Because of exogenous randomness, there is no motivation for this approach:  
\Theorem{t:SPD} guarantees a passive linearization only when $P_0=R_0$.   Moreover, numerical results using this method were not encouraging:   The resulting family of transition matrices $\{P_\zeta\}$ is extremely sensitive to  $\zeta$.  


%

\notes{Note that for the myopic design \eqref{e:BadjForm}
 holds, with $H_\zeta(x)= \util(x) \exp(\zeta \util(x)) = e^\zeta  \util(x)  $}

%
%

 \spm{axe for cdc:
 The nominal period of this load is about 36 minutes, which corresponds to a frequency of $1/(36\times 60)$~Hz. Dividing by $2\pi$ gives the nominal frequency $2.9\times 10^{-3}$~rads/sec.  This is very close to the resonant frequency seen in each of the Bode plots. 
 }


%
%
%


The linearization about $\zeta=0$ for the myopic design was similar to the IPD solution, but as seen in Figs.~\ref{f:TCL-OptBode} and \ref{f:TCL-MyopicBode},  the behaviors quickly diverge for values beyond $|\zeta|=3$. 
\spm{axe for CDC.
\Fig{f:TCL-MyopicBode} shows the Bode plots obtained from linearizations at several values of $\zeta$ for the myopic design.    The rapid change in dynamics for the linearization implies high sensitivity for the nonlinear system.   }

In conclusion, although the transfer functions for the linearizations at $\zeta=0$ are nearly identical,  in the myopic design the input-output behavior is unpredictable for $|\zeta|>3$.     The input-output behavior for IPD is much closer to a linear system for a  wider range of $\zeta$.  This is consistent with results from prior research  \cite{meybarbusyueehr15,chebusmey14,chebusmey15b}.

\section{Conclusions} 
\label{s:conc}

This paper has developed   new approaches to distributed control for demand dispatch.   There is much more work to do on algorithm design, and large-scale testing.


%
%

\spm{I updated Raginsky and our papers, except for HICSS 2016.  
I also updated  guaragwil14}

\bibliographystyle{IEEEtran}
\def\cprime{$'$}\def\cprime{$'$}

\end{document}